\documentclass[11pt]{amsart}
\usepackage{amsmath,amsthm,verbatim,amssymb,amsfonts,amscd,graphicx}
\usepackage[colorlinks, linkcolor=blue, citecolor=red]{hyperref}
\usepackage{tikz-cd} 
\usepackage{graphics}
\usepackage{mathtools}
\usepackage{mathrsfs}
\usepackage{relsize}
\usetikzlibrary{arrows,calc,matrix}
\usepackage{adjustbox}
\usepackage{paralist}
\usepackage{cases}
\usetikzlibrary{matrix}
\newtheorem{theorem}{Theorem}[section]
\newtheorem{lemma}[theorem]{Lemma}
\newtheorem{proposition}[theorem]{Proposition}
\newtheorem{corollary}[theorem]{Corollary}
\theoremstyle{definition}
\newtheorem{definition}[theorem]{Definition}
\newtheorem{example}[theorem]{Example}

\theoremstyle{remark}
\theoremstyle{conjecture}
\newtheorem*{ack}{Acknowledgement}
\newtheorem{remark}[theorem]{Remark}

\newtheorem{conjecture}[theorem]{Conjecture}
\newtheorem{question}[theorem]{Question}

\newcommand{\Ker}{\operatorname{Ker}}
\newcommand{\im}{\operatorname{Im}}
\newcommand{\rk}{\operatorname{rk}}
\newcommand{\Bigwedge}{\mathord{\adjustbox{totalheight=.85\baselineskip}{$\bigwedge$}}}

\usepackage[a4paper,
left=1in,right=1in,
bottom=1.2in,
]{geometry}
\def\z2{\mathbb{Z}_2}
\def\Z{\mathbb{Z}}
\def\p{P(m,n)}
\def\zp{\mathbb{Z}_p}
\def\zz{\mathbb{Z}_2 \oplus \mathbb{Z}_2}
\def\coh{H^*(X; \mathbb{Z}_2)}
\def\cpn{\mathbb{C}P^n}

\def\rp{\mathbb{R}P^}
\def\scp{S^m \times \mathbb{C}P^n}
\def\sph{S^{n_1} \times \cdots \times S ^{n_k}}

\begin{document}

\title[Group actions on products of Dold manifolds]{Free actions of finite  groups on products of Dold manifolds}
\author{PINKA DEY}
\address{Department of Mathematical Sciences, Indian Institute of Science Education and Research (IISER) Mohali, Knowledge City, Sector 81, SAS Nagar, Manauli (PO), Punjab 140306, India.}
\email{pinkadey11@gmail.com.}

\begin{abstract}
The Dold manifold $ P(m,n)$ is the quotient of $S^m \times \mathbb{C}P^n$ by the free involution that acts antipodally  on $ S^m $ and by complex conjugation on $ \mathbb{C}P^n $. In this paper, we investigate free actions of finite groups on products of Dold manifolds. We show that if a finite group $ G $ acts freely and mod 2 cohomologically trivially on a finite-dimensional CW-complex homotopy equivalent to ${\displaystyle \prod_{i=1}^{k} P(2m_i,n_i)}$, then $G\cong (\mathbb{Z}_2)^l$ for some $l\leq k$. This is achieved by first proving a similar assertion for $ \displaystyle \prod_{i=1}^{k} S^{2m_i} \times \mathbb{C} P^{n_i} $. 
 We also determine the possible mod 2 cohomology algebra of orbit spaces of arbitrary free involutions on Dold Manifolds, and give an application to $ \mathbb{Z}_2 $-equivariant maps.

\end{abstract}
\maketitle

\section{Introduction}

In 1925, Hopf stated the following topological spherical space form problem as follows:\par \textit{Classify all manifolds  with universal cover homeomorphic to the $n$-dimensional sphere $S^n$ for $n>1$}.\par This is equivalent to determining all finite groups that can act freely on $S^n$. Smith \cite{smith} showed that such a group $G$ must have periodic cohomology, i.e., $ G $ does not contain a subgroup of the form $ \mathbb{Z}_p \oplus \mathbb{Z}_p$ for any prime $p$.\par
A natural generalization of the above problem is to classify all finite groups that can act freely on products of finitely many spheres. More generally, one can ask a similar question for arbitrary topological spaces. This led to the concept of free rank of symmetry of a topological space, defined as follows:  
\begin{definition}
	Let $X$ be a finite-dimensional CW-complex and $p$ be a prime number. The free $p$-rank of $X$, denoted by $\textrm{frk}_p(X)$, is the largest $r$ such that $(\mathbb{Z}_p)^r$ acts freely on $X$. The smallest value of $r$ over all primes $p$ is called
	the free rank of $X$ and denoted by frk($ X $).
\end{definition} 
By Smith's result, we get that \begin{displaymath}
\textrm{frk}_p \,(S^ n)= \left\{   \begin{array}{lll}
1   & \textrm{if $n$ is odd and $ p $ is arbitrary,}\\
1 & \textrm{if $n$ is even and $ p=2$,}\\
0 & \textrm{if $n$ is even and $ p>2 $. }\\
\end{array} \right.
\end{displaymath}
For products of finitely many spheres, the following statement appears  as a question \cite[Question 7.2]{adem-browder} or as a conjecture \cite[p. 28]{adem-davis}.
\begin{conjecture}
	If $ (\mathbb{Z}_p)^r $ acts freely on $S^{n_1}\times S^{n_2} \times \cdots \times S^{n_k}  $, then $ r\leq k$.
\end{conjecture}
\noindent Carlsson \cite{carlsson-nonexistence, carlsson-p>2} proved the above conjecture for homologically trivial actions on products of equidimensional spheres. Adem and Browder \cite{adem-browder} proved that $ \textrm{frk}_p((S^n)^k) =k$ with the only remaining cases as $ p=2 $ and $ n=1, 3, 7 $. Later, Yal\c{c}in \cite{yalcin-group-extensions}  proved that $\textrm{frk}_2((S^1)^k)=k  $. The most general result which settles a stable form of the above conjecture is due to Hanke \cite[Theorem 1.3]{hanke}, which states  that
$${\textrm{frk}_p}\big(\sph\big)=k_0 \hspace{4mm}\textrm{if}\hspace{2mm} p>3(n_1+\cdots+n_k),$$
where $k_0$ is the number of odd-dimensional spheres in the product.\par
In a recent work \cite[Theorem 1.2]{okutan-yalcin}, Okutan and Yal\c{c}in proved the conjecture in the case when the the dimensions $\{n_i\}$ of the spheres are high compared to the differences $|n_i - n_j|$ between the dimensions. The general case is still open.\par
For free actions of  arbitrary finite  groups, there is a more general conjecture due to Benson-Carlson\cite{benson-carlson}.  For a finite group, define the rank of $G$, denoted by $\rk(G)$, as 
$$ \rk(G) = \textrm{ max} \:  \big\{r \;|\; (\mathbb{Z}_p)^r \leqslant G \; \textrm{for some prime $p$ } \big\} $$
and define 
$$ h(G) = \textrm{ min} \:\big\{k\;|\; G\textrm{ acts freely on } \sph \big\} .$$
\begin{conjecture}\hspace{-1mm}(Benson-Carlson)\;
	 If $G$ is a finite group then, $\textrm{rk}(G)= h(G)$.
\end{conjecture} 
Note that, $h(G)$ is well-defined since a result of Oliver \cite{oliver} states that every finite group acts freely on some products of spheres.\par
For products of even-dimensional spheres, Cusick \cite{ cusick-even-sphere, cusick-free-action-even-sphere}  proved that if a finite group $G$ acts freely on $S^{2n_1} \times \cdots \times S ^{2n_k}$ such that the induced action on $ H_*(X;\mathbb{Z}_2) $ is trivial, then $G\cong (\z2)^r$ for some $r\le k$.\par
Though an extensive amount of work has been done for spheres and their products, the free rank of many other interesting spaces is still unknown. A few results are known such as for products of lens spaces and for products of projective spaces. For lens spaces, Allday \cite[Conjecture 5.2]{allday} conjectured that 
$$ \textrm{frk}_p \,{ \big(L_p^{2n_{1}-1} \times \cdots \times L_p^{2n_{k}-1}\big)}=k .$$
Yal\c{c}in \cite{yalcin-lens} proved the equidimensional case of the above conjecture.\par
For products of real proective spaces, Cusick  \cite{cusick-realprojective} conjectured that if $( \mathbb{Z}_2)^r $ acts freely and mod 2 cohomologically trivially on $ \mathbb{R}P^{n_1}\times\cdots\times\mathbb{R}P^{n_k}, $ then $ r\leq \eta(n_1)+\cdots+\eta(n_k),$ where the function $ \eta $ on natural number is defined by \begin{equation*}
\eta(n)= \left\{   \begin{array}{ll}
0 & \textrm{if $n$ is even},\\
1 & \textrm{if $n\equiv1\!\!\!\!\pmod 4$, }\\
2 & \textrm{if $n\equiv3\!\!\!\!\pmod 4$. }\\
\end{array} \right.
\end{equation*}
He proved the conjecture when $ n_i\not\equiv3\!\pmod4  $ for all $ i $. Later, Yal\c{c}in \cite[Theorem 8.3]{yalcin-group-extensions} proved the conjecture when $ n_i $ is odd for each $ 0\leq i\leq k $. The general case is still open.\par
For products of complex projective spaces, Cusick \cite[Theorem 4.13]{cusick-euler} proved that if a finite group $G$ acts freely on $ (\mathbb{C} P^m)^k $, then $G$ is a 2-group of order atmost $ 2^k $ and the exponent of $ G $ cannot exceed $ 2k $.\par 
Viewing the product of two projective spaces as a trivial bundle, it is interesting to determine the free rank of symmetry of twisted projective space bundles over  projective spaces. Milnor and Dold manifolds are fundamental examples of such spaces.  
It is well-known that Dold and Milnor manifolds give
generators for the unoriented cobordism algebra of smooth manifolds. Therefore, determining various invariants of these manifolds is an interesting problem. The problem of determining the free 2-rank of symmetry of products of  Milnor manifolds has been investigated in \cite{msinghmilnor}, wherein some bounds have been obtained for the same.\par
In this paper, we investigate the free rank problem for products of Dold manifolds. A Dold manifold is the quotient of  $ \scp $ by a free involution (see Definition \ref{definition}) and it is denoted by $ P(m, n) $. To compute the free rank of ${\displaystyle \prod_{i=1}^{k} P(2m_i,n_i)}$, we first determine the free rank of ${\displaystyle \prod_{i=1}^{k} S^{2m_i}\times\mathbb{C}P^{n_i}}$. Since  a Dold manifold is a twisted complex projective space bundle over a real projective space, our arguments also give the free rank of  ${\displaystyle \prod_{i=1}^{k} \mathbb{R}P^{2m_i}\times\mathbb{C}P^{n_i}}$. To state our main results, we first define a function $ \mu:\mathbb{N} \to \{0,1\} $  by
\begin{equation}\label{mu}
\mu(n)= \left\{   \begin{array}{ll}
1 & \textrm{if $n$ is odd},\\
0 & \textrm{if $n$ is even.}\\
\end{array} \right.
\end{equation}\vspace{1mm}
\begin{theorem}\label{free-rank-univ-cover-introduction}
	Suppose $G$ is a finite group acting freely and mod 2 cohomologically trivially on a finite-dimensional CW-complex $X$ homotopy equivalent to  ${\displaystyle \prod_{i=1}^{k} S^{2m_i}\times\mathbb{C}P^{n_i}}$.  Then $G\cong (\mathbb{Z}_2)^l$ for some $ l $ satisfying
$$ l\leq	k_0 + \mu(n_1)+\mu(n_2)+\cdots+\mu(n_k),$$
where $ k_0 $ is  number of the positive-dimensional spheres in the product.
\end{theorem}\medskip
\begin{theorem}\label{free-rank-dold-introduction}
	Suppose $G$ is a finite group acting freely and mod 2 cohomologically trivially on a finite-dimensional CW-complex $X$ homotopy equivalent to  ${\displaystyle \prod_{i=1}^{k} P(2m_i,n_i)}$. Then $G\cong (\mathbb{Z}_2)^l$ for some $l$ satisfying
	$$ l\leq \mu(n_1)+\mu(n_2)+\cdots+\mu(n_k).$$
\end{theorem}\medskip
Since $ P(0, n)=\mathbb{C}P^n $, by taking all $ m_i =0$, we retrive the result of Cusick \cite[Theorem 4.15]{cusick-euler}.\vspace{1mm}\par
For an arbitrary free finite $ G $-CW-complex, there is a more general conjecture  due to Carlsson \cite[Conjecture I.3]{carlsson-conjecture}:
\begin{conjecture}\label{conjecture}
Suppose  $G=(\mathbb{Z}_p)^k$ and $ X $ is a finite free $G$-complex. Then $$2^k\leq \sum_{i=0}^{\dim X} \dim_{\mathbb{Z}_p}\, H_i(X;\mathbb{Z}_p).$$
\end{conjecture}
In particular, the above conjecture is  applicable for CW-complexes. A direct check shows that 
\begin{displaymath}
\sum_{i=0}^{\dim X} \dim_{\mathbb{Z}_p}\, H_i(X;\mathbb{Z}_p)\geq \left\{   \begin{array}{ll}
2^{k_0 + \mu(n_1)+\mu(n_2)+\cdots+\mu(n_k)}   & \textrm{for  $X=\displaystyle \prod_{i=1}^{k} S^{2m_i}\times\mathbb{C}P^{n_i}$},\\
2^{ \mu(n_1)+\mu(n_2)+\cdots+\mu(n_k)} & \textrm{for  $X=\displaystyle \prod_{i=1}^{k} P(2m_i,n_i)$},\\

\end{array} \right.
\end{displaymath}
for mod 2 cohomologically trivial actions. Thus, Conjecture \ref{conjecture} is verified for our spaces under these conditions.\par
We also consider the problem of determining the possible mod 2 cohomology algebra of orbit spaces of free involutions on Dold manifolds. Some results related to the cohomology algebra of orbit spaces can be found in \cite{pinka,msinghprojective,msinghlens}. In a recent paper \cite[Theorem 1]{morita}, Morita et al. determined the possible mod 2 cohomology algebra of orbit spaces of free involutions on Dold manifolds $P(1,n)$, where $n$ is an odd natural number. In this paper, we prove the general case using slightly different techniques.
\begin{theorem}
Let $G=\mathbb{Z}_2$ act freely on a finite-dimensional CW-complex homotopy equivalent to the Dold manifold $P(m,n)$ with a trivial action on mod 2 cohomology. Then $H^*\big(P(m,n)/G; \mathbb{Z}_2\big)$ is isomorphic to one of the following graded algebras:
\begin{enumerate}
\item  $\mathbb{Z}_2[x,y,z]/ \big\langle x^2, y^{\frac{m+1}{2}}, z^{n+1} \big\rangle,$\\
where $\deg(x)=1$, $\deg(y)=2$, $\deg(z)=2$ and $m$ is odd;
\item  $\mathbb{Z}_2[{x},{y},{z}]/\big\langle {f}, {g}, {z}^{\frac{n+1}{2}}+h\big\rangle,$\\
where $\deg({x})=1$, $\deg({y})=1$, $\deg({z})=4$, $n$ is odd,
$$f= \big(x^{m+1}+ \alpha_1 x^my+\alpha_2x^{m-1}y^2\big) \hspace{2mm} \textrm{and}\hspace{2mm} g= \big(y^3+ \beta_1 xy^2+\beta_2x^2y\big) ,$$ where $\alpha{_i}$, $\beta_j\in \z2$, and $  h \in \z2[{x}, {y}, {z }]$ is either a zero polynomial or it is a homogeneous polynomial of degree $ 2n+2 $ with the highest power of $ {z} $ is less than or equal to $ \frac{n-1}{2} $.
\end{enumerate}
\end{theorem}
This paper is organized as follows. In Section \ref{sec2}, we recall the definition and cohomology of Dold manifolds. In Section \ref{sec3}, we construct free involutions on these manifolds.  In Section \ref{sec4}, we study the free rank of symmetry of products of these manifolds. In Section \ref{sec5}, we determine the possible mod 2 cohomology algebra of orbit spaces of free involutions on Dold manifolds. In Section \ref{sec6}, we give an application to $ \z2 $-equivariant maps.\par 
Throughout the paper, by $\mathbb{Z}_p$ we denote the cyclic group of order $p$. For topological spaces $ X $ and $ Y $, $ X\simeq Y $ means that they are homotopy equivalent.

\bigskip

\section{Definition and cohomology of Dold manifolds}\label{sec2}
\begin{definition}\label{definition} Let $m, n$ be non-negative integers such that $m+n>0$.  The Dold manifold $P(m,n)$  is a closed connected smooth manifold of dimension $m+2n$, obtained as the quotient of $S^m \times \mathbb{C}P^n$ under the free involution  
$$\Big((x_0, x_1, \dots, x_m),\,[z_0, z_1, \dots, z_n] \Big)  \mapsto \Big((-x_0, -x_1, \dots, -x_m),\,[\overline{z}_0, \overline{z}_1, \dots, \overline{z}_n] \Big).$$
\end{definition}
It was shown by Dold \cite{dold} that, for suitable values of $m$ and $n$, the cobordism classes of $P(m,n)$ serve as generators in odd degrees for the unoriented cobordism algebra of smooth manifolds.\par 
The projection $\scp \rightarrow S^m$ induces the map $$p:P(m,n)\rightarrow \rp m,$$ which is a locally trivial fiber bundle with fiber  $\cpn$.  It admits a cross-section $s: \rp m \rightarrow P(m,n)$ defined as $s([x])=\big[(x, \: [1,0,\dots,0])\big]$.  Consider the cohomology classes $a=p^*(x)\in H^1(P(m,n);\z2)$, where $x$ is the generator of $H^1(\rp m; \z2)$, and $b\in H^2(P(m,n);\z2)$, where $ b $ is characterized by the property that the restriction to a fiber is non-trivial and $s^*(b)=0$.
The classes $a\in H^1(P(m,n);\z2)$ and $b\in H^2(P(m,n);\z2)$ generate $H^*(P(m,n);\z2)$. In particular, Dold \cite{dold} proved that  $$H^*(P(m,n); \mathbb{Z}_2) \cong \mathbb{Z}_2[a,b]/\big\langle a^{m+1}, b^{n+1} \big\rangle,$$ where $\deg (a)=1$ and $\deg (b) =2$.\par
Dold manifolds have been well-studied in the past. Some results related to group actions on these manifolds can be found in \cite{morita,peltier,ucci}. In a very recent work \cite{nath-sankaran}, Nath and Sankaran defined generalized  Dold manifolds and obtained results on stable  parallelizability of such manifolds.

\bigskip

\section{Free involutions on  Dold manifolds}\label{sec3}
In this section, we give some examples of free involutions on Dold manifolds. By Euler characteristic argument one can see that a Dold manifold $ P(m, n) $ admits a free involution if at least one of $ m $ and $ n $ is odd.\par
(i) When $n$ is odd: Define $T_1:\mathbb{C}P^n \rightarrow \mathbb{C}P^n $ as the free involution given by $$ T_1 \Big([z_0, z_1, \dots, z_{n-1}, z_n ]\Big) = [-\overline{z}_1, \overline{z}_0, \dots, -\overline{z}_n, \overline{z}_{n-1}].$$
Then $$id \times T_1 : S^m \times \mathbb{C}P^n \rightarrow S^m \times \mathbb{C}P^n $$
gives a free involution on $S^m \times \mathbb{C}P^n$. Since $T_1$ commutes with conjugation, we get a free involution on $P(m,n)$ by passing to the quotient.\vspace{1mm} \par
(ii)\label{t2} When $m$ is odd: Let $T_2$ be the action on $S^m \times \cpn$ defined  by 
$$T_2\Big((x_0, x_1, \dots, x_{m-1}, x_{m}), \,[z_0, z_1, \dots, z_{n}]\Big) =\Big((-x_1, x_0, \dots, -x_{m}, x_{m-1}), \,[z_0, z_1, \dots, z_{n}]\Big).$$ Note that $T_2$ is free and it commutes with the action defining the Dold manifold. Hence, we get a free involution on $P(m,n)$ by passing to the quotient.\vspace{1mm} \par
(iii) When both $m$ and $n$ are odd: In this case, define the action  $T_3$ on $S^m \times \cpn$ by 
$$T_3\Big((x_0, x_1, \dots, x_{m-1}, x_{m}), \,[z_0, z_1, \dots, z_{n}]\Big) =\Big((-x_1, x_0, \dots, -x_{m}, x_{m-1}),\,[-\overline{z}_1, \overline{z}_0, \dots, -\overline{z}_n, \overline{z}_{n-1}]\Big).$$ This gives another free involution on $P(m,n)$ by passing to the quotient.

\bigskip

\section{Free rank of symmetry}\label{sec4}
In this section, we prove  some results regarding the free rank of symmetry of ${\displaystyle \prod_{i=1}^{k} P(2m_i,n_i)}$. For this we first determine the free rank of symmetry of its universal cover ${\displaystyle \prod_{i=1}^{k} S^{2m_i}\times\mathbb{C}P^{n_i}}$. 
 Our main computational tool is the Leray-Serre spectral sequence associated to the Borel fibration.  For details, we refer the reader to \cite{allday-puppe,mccleary}.\medskip
\begin{lemma}\label{lemma-zp-trivial-imply-z-trivial}
Let $ p $ be an odd prime. If\, $ \zp$ acts freely and mod 2 cohomologically trivially on a finite-dimensional CW-complex $X$ homotopy equivalent to ${\displaystyle \prod_{i=1}^{k} S^{2m_i}\times\mathbb{C}P^{n_i}}$, then the induced action on  $ H^*(X;\mathbb{Z}) $ is  trivial.
\end{lemma}

\begin{proof}
As the integral cohomology groups of $ X $ are torsion-free, we get $$  H^k(X;\z2)\cong  H^k(X;\Z) \otimes \z2 .$$  Let $ g$ be a generator of $ \zp $ and  $ \phi: \Z \to \z2 $ be the non-trivial homomorphism. Then we have the following commutative diagram:

\[\begin{tikzcd}[column sep=4em, row sep=2em]
H^k(X;\Z) \ar{r} {g_\Z^*} \ar[swap]{d}{\phi^*} & H^k(X;\Z) \ar{d}{\phi^*} \\
H^k(X;\z2) \ar[swap]{r}{g_{\z2}^*} &  H^k(X;\z2), 
\end{tikzcd}
\]
\noindent where  $ g_\Z^* $ and $ g_{\z2}^* $ are the induced maps on cohomology with $ \Z $ and $ \z2 $-coefficients, respectively.  Let $ A\in \mathrm{GL}(n,\Z) $ be the matrix corresponding to ${g_{\small {\Z}}^*}$. Then the matrix $ \overline{A} $ corresponding to ${g_{\small{\z2}}^*}$ is obtained by reducing all the entries of $ A $ modulo 2. By assumption  $\overline{A}=I$, the identity matrix. Consequently, $ A=2M+I$ for some matrix $ M $ and $ A^p=I $. Thus,
 \begin{equation}\label{equ-matrix}
 A^p-I = (A-I) \hspace{1mm}(A^{p-1} + A^{p-2} + \cdots +  I  )=0,
 \end{equation}  
and
$$ (A^{p-1} + A^{p-2} + \cdots + I  )= (2M + I)^{p-1} +  (2M+I)^{p-2} + \cdots  + I= 2N + I,$$ for some matrix $ N $.  Since det$ (2N + I) $ is odd,  
$ (A^{p-1} + A^{p-2}+ \cdots + I  )$ is invertible. It follows from equation \eqref{equ-matrix} that $ A=I $.
\end{proof}
By using Lefschetz fixed point theorem we get the following result.   
\begin{lemma}\label{lemma-zp-not-action-univ-Z-cohomlogy}
	If $ p $ is an odd prime, then  $\mathbb{Z}_p$ cannot act freely and $ \mathbb{Z} $-cohomologically trivially on a finite-dimensional CW-complex $X$ homotopy equivalent to  ${\displaystyle \prod_{i=1}^{k} S^{2m_i}\times\mathbb{C}P^{n_i}}$.
	
\end{lemma}

Lemma \ref{lemma-zp-trivial-imply-z-trivial} and Lemma  \ref{lemma-zp-not-action-univ-Z-cohomlogy} together imply the following result.
\begin{lemma}\label{lemma-zp}
	If $ p $ is an odd prime, then $\mathbb{Z}_p$ cannot act freely and mod 2 cohomologically trivially on a finite-dimensional CW-complex $X$ homotopy equivalent to  ${\displaystyle \prod_{i=1}^{k} S^{2m_i}\times\mathbb{C}P^{n_i}}$.
	\end{lemma}
\begin{lemma}\label{lemma-disjoint-union}
A finite group $ G $ acts freely and mod 2 cohomologically trivially on the disjoint union $\displaystyle {\coprod _{i=1}^{k} X}$ of a space $ X $ if and only if it acts freely and mod 2 cohomologically trivially on $ X $.
\end{lemma}
\begin{proof}
	Suppose $ G $ acts freely and mod 2 cohomologically trivially on $ X $. Then we can define a free $ G $-action on $\displaystyle {\coprod _{i=1}^{k} X}=\displaystyle {\bigcup _{i=1}^{k}\,X \times \{i\}}$ by setting $$ g\cdot(x, i)=(gx, i) .$$
	Since the cohomology algebra of a disjoint union is the direct product of the cohomology algebra of the components, it follows that the induced action of $ G $ on mod 2 cohomology algebra is also trivial. Conversely, if $ G $ acts freely and mod 2 cohomologically trivially on $\displaystyle {\coprod _{i=1}^{k} X}$, then the action preserves each component and acts as desired.

\end{proof}

\medskip
Now we prove our main results. For the definition of the function $ \mu $ see equation \eqref{mu}.
\begin{theorem}\label{univ-cover}
Let $G$ be a finite group acting freely and mod 2 cohomologically trivially on a finite-dimensional CW-complex $X$ homotopy equivalent to ${\displaystyle \prod_{i=1}^{k} S^{2m_i}\times\mathbb{C}P^{n_i}}$.  Then $G\cong (\mathbb{Z}_2)^l$ for some $ l $ satisfying 
$$ l\leq	k_0 + \mu(n_1)+\mu(n_2)+\cdots+\mu(n_k),$$
where $ k_0 $ is the number of positive-dimensional spheres in the product.

\end{theorem}
\begin{proof}
	First, we consider the case where each $m_i>0  $. We prove the assertion by first showing that $\mathbb{Z}_4$ does not act freely and mod 2 cohomologically trivially on $X\simeq{\displaystyle \prod_{i=1}^{k} S^{2m_i}\times\mathbb{C}P^{n_i}}$. On the contrary, suppose  $\mathbb{Z}_4$ acts freely and mod 2 cohomologically trivially on $ X $. By K\"{u}nneth formula we get $$H^*(X; \mathbb{Z}_2)=\mathbb{Z}_2[a_1,\dots, a_k]/\big\langle a_1^{n_1+1},\dots ,a_k^{n_k+1}\big\rangle \otimes \Bigwedge \big(b_1,\dots, b_k\big),$$  where $\deg(a_i)=2$ and  $\deg(b_i)=2m_i$. Consider the Leray-Serre spectral sequence $ (E_r^{*,*},d_r) $ associated to the Borel fibration $$X\stackrel{}{\hookrightarrow} X_{\mathbb{Z}_4} \stackrel{\pi}{\longrightarrow} B_{\mathbb{Z}_4}.$$ 
	The $ E_2$-page of this spectral sequence looks like 
	\[   E_2^{*,*}  \cong  H^*(B_{\mathbb{Z}_4}; H^*(X;\mathbb{Z}_2)).\]
	Since the induced action on mod 2 cohomology is trivial, the system of local coefficients is simple.  Thus, we have
	$$
	E_2^{*,*}  \cong  H^*(B_{\mathbb{Z}_4};\mathbb{Z}_2) \otimes H^*(X; \mathbb{Z}_2).
	$$
	 Consider the subgroup inclusion $  i: \z2 \hookrightarrow \mathbb{Z}_4 $. By restricting the $ \mathbb{Z}_4 $ action to $ \z2 $, we get a free $ \z2 $ action on $ X $. Then consider the Leray-Serre spectral sequence $(\widetilde{E } _r^{*,*}, \widetilde{d}_r)$ associated to the Borel fibration $$X\stackrel{}{\hookrightarrow} X_{\mathbb{Z}_2} \stackrel{\pi}{\longrightarrow} B_{\mathbb{Z}_2}.$$
	The naturality of the Leray-Serre spectral sequence gives the following commutative diagram:
	\[\begin{tikzcd}[column sep=4em, row sep=2em]
	E_r^{k,l} \ar{r} {f^*} \ar[swap]{d}{d_r} & \widetilde{E }_r^{k,l} \ar{d}{\widetilde{d}_r} \\
	E_r^{k+r,l-r+1} \ar[swap]{r}{f^*} &  \widetilde{E }_r^{k+r,l-r+1},
	\end{tikzcd}
	\]
	\noindent where at the $ E_2 $-page the map $ f^*: (E_r^{*,*},d_r) \to  (\widetilde{E } _r^{*,*}, \widetilde{d}_r)$ is defined by $$f^*= i^* \otimes id: H^*(B_{\mathbb{Z}_4}; \z2) \otimes \coh \to H^*(B_{\z2};\z2)\otimes \coh.$$ 
We know that, 
$$H^*(B_{\mathbb{Z}_4}; \mathbb{Z}_2)\cong \mathbb{Z}_2[x]\otimes\Bigwedge(y),$$ where $\deg(x)=2$ and $\deg(y)=1 $. We also recall that $H^*(B_{\mathbb{Z}_2}; \mathbb{Z}_2)\cong \mathbb{Z}_2[t],$ where $\deg(t)=1$. It can be seen that the map $ i^*: H^*(B_{\mathbb{Z}_4}; \mathbb{Z}_2) \to H^*(B_{\mathbb{Z}_2}; \mathbb{Z}_2)$ is given by $ i^*(x)=t^2 $ and $ i^*(y)=0 $. Since $i^*(y)=0$, it follows that  $f^*$ is zero on ${E_r^{odd, *}}$. 
Since $H^*(X; \mathbb{Z}_2)$ has non-zero cohomology only in even dimension, it follows that the non-zero differentials lie in $E_r$-page, where $r$ is odd.  Consider an element $z\in H^*(X; \mathbb{Z}_2)$. 
Then $$\widetilde{d}_{2n+1}(1\otimes z)=\widetilde{d}_{2n+1}f^*(1\otimes z) =f^*d_{2n+1}(1\otimes z).$$ As $d_{2n+1}(1\otimes z) \in E_{2n+1}^{odd, *}$, it follows that $\widetilde{d}_{2n+1}(1\otimes z)=0$. Since this is true for any $z\in H^*(X; \mathbb{Z}_2)$ and for any differential $\widetilde{d}_r$, it follows that the corresponding spectral sequence collapses at the $E_2$-page. This contradicts the fact that $X/\mathbb{Z}_2$ is finite-dimensional. Thus, we have proved that $\mathbb{Z}_4$ cannot act freely on $X$.\par
If $ G $ is not a finite 2-group, then there exists a prime $ p> 2$ such that $ p$ divides $|G|$. Consequently, $ \zp $ acts freely and mod 2 cohomologically trivially on $ X $, which contradicts Lemma \ref{lemma-zp}. Therefore, $ G $ is a finite 2-group.\par
Now if $ n_i\equiv 1\pmod 4 $ for all $ i $, then the desired bound on the rank of $ G $ follows from  Euler characteristic argument. But, if  $n_i\equiv 3\pmod 4$ for some $ i $, then the Euler characteristic fails to give exact bound on the rank of $G$. To derive the exact bound, we first recall some properties of Lefschetz number.  For a map $ f:X \to X $, we can define the  Lefschetz number $ \tau(f) $ as $$ \tau(f)= \sum_{i} (-1)^i\; tr(f^*:H^n(X; \mathbb{Z})\to H^n(X; \mathbb{Z}) ).$$
Suppose $ f^* $ is an automorphism of $ H^*(X; \mathbb{Z}) $ such that $ f^*(a_i)=\xi_i a_i $ and $ f^*(b_i)=\lambda_i b_i $, where $ a_i, b_i $ are the generators of $ H^*(X;\mathbb{Z}) $ and $ \xi_i, \lambda_i \in \mathbb{Z}$. Then its Lefschetz number is precisely the following
  \begin{equation}\label{equ-lefschetz}
 \tau(f)=\prod_{i=1}^{k}(1 + \xi_i+\xi_i^2+\cdots+\xi_i^{n_i})(1+\lambda_i).
 \end{equation}\par
Let $ g$ be an element of  $G $ and $ g^*_{\mathbb{Z}} $ be the induced map on $ H^*(X; \mathbb{Z}) $. Since $ G $ acts freely, the Lefschetz number of $ g^*_{\mathbb{Z}} $ must vanish for all $ g\ne 1 $. This shows that  $ G $ is isomorphic to some subgroup of $ Aut(H^*(X; \mathbb{Z})) $. Notice that, an element of $  Aut(H^*(X;\mathbb{Z})) $ is completely determined by its value on the generators $ a_1, a_2, \dots, a_k $ and $ b_1, b_2, \dots ,b_k $. Since $  g^*_\mathbb{Z} $ is an automorphism, it preserves degrees as well as cup-length, where for an element $ x\in H^*(X; \mathbb{Z}) $, the cup-length of $ x $ is the greatest integer $ n $ such that $ x^n \ne 0$. Thus, $ g^*_\mathbb{Z} $ can only permute between those generators which have same degrees and same cup-length (this follows because the cup-length of a sum of the generators is the sum of the cup-lengths of the individual generators). In other words, $ g^*_\mathbb{Z}(a_i)= \pm a_{\sigma(i)} $ and $ g^*_\mathbb{Z}(b_i)= \pm b_{\gamma(i)} $, where  $ \sigma,  \gamma \in S_k$ (the symmetric group on $ k $ elements) and $ a_\sigma(i)  $ has same cuplength as $ a_i $. Let  $  g^*_{\mathbb{Z}_2} $ be the induced map on $ H^*(X; \mathbb{Z}_2) $. By the naturality of the cohomology  functor, we get  the following commutative diagram:

\[\begin{tikzcd}[column sep=4em, row sep=2em]
H^k(X;\Z) \ar{r} {g_\Z^*} \ar[swap]{d}{} & H^k(X;\Z) \ar{d}{} \\
H^k(X;\z2) \ar[swap]{r}{g_{\z2}^*} &  H^k(X;\z2). 
\end{tikzcd}
\]
\noindent
Since we have assumed that  $ G $ acts mod 2 cohomologically trivially, we get   $ g^*_\mathbb{Z}(a_i)= \pm a_{i} $ and $ g^*_\mathbb{Z}(b_i)= \pm b_{i} $. This shows that 
$ G$ is a subgroup of  $(\mathbb{Z}_2)^{2k} $. The Lefschetz number should be zero corresponding to each non-identity element of $ G $. From (\ref{equ-lefschetz}) we get that, $ \tau(f)=0 $ only if either  $ \lambda_i=-1 $ for some $ i $ or $1 + \xi_j+\xi_j^2+\cdots+\xi_j^{n_j}=0  $ for some $ j $. The last case is possible only if $ n_j $ is odd and in that case we get $ \xi_j=-1 $. This implies that $$ G\subseteq (\mathbb{Z}_2)^{k + \mu(n_1)+\mu(n_2)+\cdots+\mu(n_k)} ,$$ where $ \mu(n)=0 $ if $ n $ is even and $ \mu(n)=1 $ if $ n $ is odd (note that if some $n_j  $ is even, then an automorphism $ g^* $ that maps  $ a_j \mapsto -a_j $ and fixes all other generators of $ H^*(X;\mathbb{Z}) $ is not possible). Consequently, $G\cong (\mathbb{Z}_2)^l$ for some $ l $ satisfying 
$$ l\leq	k + \mu(n_1)+\mu(n_2)+\cdots+\mu(n_k).$$\par
Next, we consider the case where $ m_i =0$ for some $ i $. Let $ k_0 $ be the number of positive-dimensional spheres in the product. Note that, ${\displaystyle \prod_{i=1}^{k} S^{2m_i}\times\mathbb{C}P^{n_i}}$ can be written as disjoint union of $ 2^{k-k_0} $ copies of $\Big({\displaystyle{ \prod_{m_i>0} S^{2m_i} \times \,\prod_{i=1}^{k} \mathbb{C}P^{n_i}  }}\Big)$. By  Lemma \ref{lemma-disjoint-union}, it is enough to prove the result for $\Big({\displaystyle{ \prod_{m_i>0} S^{2m_i} \times \,\prod_{i=1}^{k} \mathbb{C}P^{n_i}  }}\Big)$. Applying similar argument as in the previous case, we get the desired result.
\end{proof}

\begin{corollary}
	
	   For mod 2 cohomologically trivial actions
	  
	  \begin{equation*}
	  \textrm{frk}_p(\, {\displaystyle \prod_{i=1}^{k} S^{2m_i}\times\mathbb{C}P^{n_i}} \,)= \left\{   \begin{array}{ll}
	  k_0 + \mu(n_1)+\mu(n_2)+\cdots+\mu(n_k) & \textrm{if $p=2$},\\
	  0 & \textrm{if $p>2$,}\\
	  \end{array} \right.
	  \end{equation*}
	  where $ k_0 $ is the number of positive-dimensional spheres in the product.

\end{corollary}
\vspace{.5mm}

\begin{remark}
	
If we take arbitrary dimensional spheres in the statement of Theorem \ref{univ-cover}, then our method does not work. Here we would like to emphasize that the problem of determining the free rank for products of different dimensional spheres is a long standing open problem. 
\end{remark}
\vspace{1mm}
In view of the preceding remarks, we would like to propose the following conjecture.
\begin{conjecture}\label{conjectureforuniv}
	If $ (\z2)^l$ acts freely and mod 2 cohomologically trivially on a finite-dimensional CW-complex   $X\simeq{\displaystyle \prod_{i=1}^{k} S^{m_i}\times\mathbb{C}P^{n_i}}$ where each $ m_i>0 $, then $$l \leq k + \mu(n_1)+\cdots+ \mu(n_k). $$
\end{conjecture}

\medskip
Next we prove a result related to the free rank of Dold manifolds. 
\begin{theorem}\label{doldfreerank}
Suppose $G$ is a finite group acting freely and mod 2 cohomologically trivially on a finite-dimensional CW-complex $X$ homotopy equivalent to ${\displaystyle \prod_{i=1}^{k} P(2m_i,n_i)}$. Then $G\cong (\mathbb{Z}_2)^l$ for some $l$ satisfying
$$ l\leq \mu(n_1)+\mu(n_2)+\cdots+\mu(n_k).$$
\end{theorem}
\begin{proof}
First, we prove that $\mathbb{Z}_4$ cannot act freely and mod 2 cohomologically trivially on $X\simeq{\displaystyle \prod_{i=1}^{k} P(2m_i,n_i)}$. On the contrary, suppose $\mathbb{Z}_4$ acts freely and mod 2 cohomologically trivially on $X$ and let $Y$ denote its orbit space. Then there is an exact sequence
$$ 1\to (\mathbb{Z}_2)^k\to \pi_1(Y)\to \mathbb{Z}_4\rightarrow 1.$$
It follows from the exactness of the sequence that $\pi_1(Y)$ is a 2-group. Also, $\pi_1(Y)$ acts freely on the universal cover ${\displaystyle \prod_{i=1}^{k} S^{2m_i}\times\mathbb{C}P^{n_i}}$. By Theorem \ref{univ-cover}, $ \pi_1(Y) $ must be an elementary abelian 2-group and $\pi_1(Y)\cong(\mathbb{Z}_2)^{k+2}$. But this is not possible since $(\mathbb{Z}_2)^{k+2}\big/(\mathbb{Z}_2)^{k}$ is not isomorphic to $ \mathbb{Z}_4$. This implies that $\mathbb{Z}_4$ cannot act freely on ${\displaystyle \prod_{i=1}^{k} P(2m_i,n_i)}$. Similar proof also shows that there is no free  $ \zp $-action on ${\displaystyle \prod_{i=1}^{k} P(2m_i,n_i)}$ for odd primes $ p $.  Then using Theorem \ref{univ-cover}, one sees that $G\cong (\mathbb{Z}_2)^l$, where $ l $ is as desired.  
\end{proof}
\medskip
Examples from Section \ref{sec3} show that the preceding bound in Theorem \ref{doldfreerank} is sharp. As a consequence we get the following result on free rank.
\begin{corollary}

	For mod 2 cohomologically trivial actions
	\begin{equation*}
	\textrm{frk}_p\,(\,{\displaystyle \prod_{i=1}^{k} P(2m_i, n_i)} \,)= \left\{   \begin{array}{ll}
	\mu(n_1)+\mu(n_2)+\cdots+\mu(n_k) & \textrm{if $p=2$},\\
	0 & \textrm{if $p>2$.}\\
	\end{array} \right.
	\end{equation*}
		
\end{corollary}
\begin{remark}
Note  that, the mod 2 cohomology algebra of  $\mathbb{R}P^m \times \mathbb{C}P^n$ and that of $P(m,n)$ are isomorphic. Also, they have the same universal cover $S^m \times \mathbb{C}P^n$. Hence, the proof of  Theorem \ref{doldfreerank} also works for $\displaystyle{\prod_{i=1}^{k} \mathbb{R}P^{2m_i} \times \mathbb{C}P^{n_i}}$. 
\end{remark}

 Taking into account  Theorem \ref{doldfreerank} and the free 2-rank of $ \prod_{i=1}^{k}\mathbb{R}P^{m_i} $, we conclude this section with the following:
\begin{question}
	Does the equality  \[\textrm{frk}_2\,(\prod_{i=1}^k P(m_i, n_i))= \eta(m_1)+\cdots+\eta(m_k)+ \mu(n_1)+\cdots+\mu(n_k) \] hold? Here, $ \mu(n_i)=1 $ for $ n_i $ is odd and $ \mu(n_i)=0 $ for $ n_i $ is even and $ \eta $ is defined by
 \begin{equation*}
	\eta(m)= \left\{   \begin{array}{ll}
	0 & \textrm{if $m$ is even},\\
	1 & \textrm{if $m\equiv1\!\!\!\!\pmod 4$, }\\
	2 & \textrm{if $m\equiv3\!\!\!\!\pmod 4$. }\\
	\end{array} \right.
	\end{equation*}
\end{question}
\bigskip
\section{Cohomology algebra of orbit spaces of free involutions on Dold manifolds }\label{sec5}

 In this section, we determine the possible mod 2 cohomology algebra of orbit spaces of arbitrary free involutions on Dold manifolds. 
We require the following result in the proof of the main theorem.
\begin{lemma}\label{lemma}
	Let $\z2 $ act freely on $ \scp $. Then the induced action on $ H^*(\scp;\z2) $ is trivial.
\end{lemma}
\begin{proof}
	
	Recall that,
	$$H^*(S^m\times \mathbb{C}P^n; \Z) \cong \Z[x,y]/\big\langle x^{2}, y^{\hspace{.3mm}n+1} \big\rangle,$$ where $\deg (x)=m$ and $\deg (y) =2$.  Suppose $ \z2 $ acts freely on $\scp  $ and let $ g $ be the generator of $ \z2 $. Let $ g_\Z^* $ be the induced action on $ H^*(\scp;\Z) $ corresponding to $ g $. If $ g_\Z^* $ is non-trivial then it has one of the following possibilities:
	\begin{enumerate}
		\item[]  (i)  $ x \mapsto -x $ and $ y\mapsto y ,$
		\item[] (ii) $ x \mapsto x $ and $ y\mapsto -y ,$ 
		\item[] (iii)  $ x \mapsto -x $ and $ y\mapsto -y $.
	\end{enumerate}
	  Since $ H^k(\scp;\Z) $ is torsion-free for any $ k\geq 0 $, we get $$  H^k(\scp;\z2)\cong  H^k(\scp;\Z) \otimes \z2 .$$  Let $ \phi: \Z \to \z2$ be the non-trivial homomorphism. Corresponding to $ \phi $ and $ g $ we get the following commutative diagram:
	
	\[\begin{tikzcd}
	H^k(\scp;\Z) \ar{r} {g_\Z^*} \ar[swap]{d}{\phi^*} & H^k(\scp;\Z) \ar{d}{\phi^*} \\
	H^k(\scp;\z2) \ar[swap]{r}{g_{\z2}^*} &  H^k(\scp;\z2), 
	\end{tikzcd}
	\]
	where $ g_\Z^* $ and $ g_{\z2}^* $ are the induced maps on cohomology in $ \Z $ and $ \z2 $-coefficients, respectively. This shows that $ g_{\z2}^* =g_\Z^* \otimes \z2$ and hence  we get $ g_{\z2}^* $ is the identity map.
\end{proof}
\medskip
\begin{remark}
	The preceding lemma shows that if $\zz$ acts freely on $\scp$, then the induced action on mod 2 cohomology is trivial.
\end{remark}	
	
\begin{theorem}\label{coho-dold}
Let $G=\mathbb{Z}_2$ act freely on a finite-dimensional CW-complex homotopy equivalent to the Dold manifold $P(m,n)$ with a trivial action on mod 2 cohomology. Then $H^*\big(P(m,n)/G; \mathbb{Z}_2\big)$ is isomorphic to one of the following graded algebras:
\begin{enumerate}\label{case1}
\item $\mathbb{Z}_2[x, y, z]/ \big\langle x^2, y^{\frac{m+1}{2}}, z^{n+1} \big\rangle,$
where $\deg(x)=1$, $\deg(y)=2$, $\deg(z)=2$ and $m$ is odd ;
\item \label{case2}  $\mathbb{Z}_2[{x}, {y}, {z}]/\big\langle {f}, {g}, {z}^{\frac{n+1}{2}}+h\big\rangle,$\\
where $\deg({x})=1$, $\deg({y})=1$, $\deg({z})=4$, $n$ is odd,
$$f= \big(x^{m+1}+ \alpha_1 x^my+\alpha_2x^{m-1}y^2\big) \hspace{2mm} \textrm{and}\hspace{2mm} g= \big(y^3+ \beta_1 xy^2+\beta_2x^2y\big) ,$$
where $\alpha{_i}$, $\beta_j\in \z2$, and  $  h \in \z2[{x}, {y}, {z }]$ is either a zero polynomial or it is a homogeneous polynomial of degree $ 2n+2 $ with the highest power of $ {z} $ is less than or equal to $ \frac{n-1}{2} $.
\end{enumerate}
\end{theorem}

\begin{proof}
Recall that, the cohomology algebra of the Dold manifold is $$H^*\big(P(m,n); \mathbb{Z}_2\big) \cong \mathbb{Z}_2[a,b]/\big\langle a^{m+1}, b^{n+1} \big\rangle,$$ where $\deg (a)=1$ and $\deg (b) =2$. Consider an arbitrary free involution on $P(m,n)$ and let $M$ denote its orbit space. Note that, the fundamental group of $M$ can be either $\mathbb{Z}_4$ or $\mathbb{Z}_2\oplus \mathbb{Z}_2$, and it acts on the universal cover $S^m \times \mathbb{C}P^n$ of $P(m,n)$.\par 
{Case (1)}: First, consider the case $\pi_1(M)=\mathbb{Z}_4 $. It follows that $H^1(M;\mathbb{Z}_2)= \mathbb{Z}_2$. Consider the Leray-Serre spectral sequence associated to the Borel fibration $$P(m,n)\stackrel{i}{\hookrightarrow} M \stackrel{\pi}{\longrightarrow} B_{\mathbb{Z}_2}.$$ 
Since we have assumed that $ \pi_1(B_{\z2})=\z2 $ acts trivially on $ H^*(\p;\z2) $, the above fibration has a simple system of local coefficients. Hence, the spectral sequence has the form $$E_2^{p,q}\cong H^p(B_{\z2}; \mathbb{Z}_2) \otimes H^q(\p;\mathbb{Z}_2).$$ Since $ M $ is finite-dimensional, the spectral sequence does not degenerate at the $E_2$-page.  As $H^1(M;\mathbb{Z}_2)= \mathbb{Z}_2$, the only possibility of the differential is $d_2(1\otimes a)= t^2 \otimes 1$ and $d_2(1\otimes b)= 0$. This implies $m$ must be odd. By computing the ranks of $\Ker d_2$ and $\im  d_2$, we get rk$(E_3^{k,l})$ = 0 for all $k \geq 2$. Thus, the spectral sequence collapses at the $E_3$-page and we get that $ E_\infty^{*,*} \cong E_3^{*,*}$. Consequently, $$H^k(M;\z2)\cong \bigoplus_{i+j=k}E_{\infty}^{i,j}= E_{\infty}^{0,k} \oplus E_{\infty}^{1,k-1}$$ for  $0 \leq k \leq m+2n$.\par 
We see that ${ 1 \otimes a^2 } \in  E_2^{0,2}$  and $1 \otimes b \in E_2^{0,2}$  are permanent   
cocycles. Let $i^*$ and $\pi^*$ be  the homomorphisms induced on cohomology, and they are precisely the edge maps (see \cite{mccleary}, Theorem 5.9). Choose $y \in H^2(M;\mathbb{Z}_2)$ such that $i^*(y)=a^2$. Then we have $y^{\frac{m+1}{2}}=0$. Similarly, choose $z \in H^2(M; \mathbb{Z}_2)$ such that $i^*(z)= b$. Then we have $z^{n+1}=0$. Let $x=\pi^*(t) \in E_\infty^{1,0} \subseteq H^1(M;\z2) $ be determined by $t \otimes 1 \in E_2^{1,0}$. As $E_\infty^{2,0} = 0$, we have $x^2 = 0$. In this case, we obtain
\begin{align}\label{z4-coho}
H^*(M;\mathbb{Z}_2) \cong \mathbb{Z}_2[x,y,z]/\big\langle x^2, y^{\frac{m+1}{2}}, z^{n+1}\big\rangle,
\end{align}
where $\deg(x)=1$, $\deg(y)=2$ and $\deg(z)=2$.
\medskip

{Case (2)}: Next, consider the case $\pi_1(M)=\mathbb{Z}_2 \oplus\mathbb{Z}_2$. In this case, it is not possible to determine the cohomology algebra of the orbit spaces directly from the Leray-Serre spectral sequence associated to the Borel fibration $$P(m,n)\stackrel{i}{\hookrightarrow} M \stackrel{\pi}{\longrightarrow} B_{\mathbb{Z}_2}.$$ So we proceed as follows. First, note that $\mathbb{Z}_2 \oplus\mathbb{Z}_2$ acts freely on the universal cover $S^m \times \mathbb{C}P^n$. Consider only those free $\mathbb{Z}_2 \oplus\mathbb{Z}_2$-actions on $S^m \times \mathbb{C}P^n$ that factor through the Dold manifold. Let $\alpha$ and $\beta$ be the generators of $\mathbb{Z}_2 \oplus\mathbb{Z}_2$. Then both $ \alpha $ and $ \beta $ give rise to some free involutions on $ \scp $. Let  $\alpha$ act on $S^m \times \mathbb{C}P^n$ by the free involution defining the Dold manifold, and the action of $\beta$ on 
$S^m \times \,\mathbb{C}P^n$ is an arbitrary free involution commuting with $\alpha$. By defining the action as above we get all possible free $\mathbb{Z}_2$-actions on the Dold manifold. Consider the subgroup $\mathbb{Z}_2$ generated by $\alpha$. There is an inclusion $i: \mathbb{Z}_2\stackrel{}{\hookrightarrow} \mathbb{Z}_2 \oplus\mathbb{Z}_2 $ given by $ \alpha \rightarrow (\alpha,0)$. This gives a map between the classifying spaces  $$B_{\mathbb{Z}_2} \rightarrow B_{\:\mathbb{Z}_2 \:\oplus\:\mathbb{Z}_2}.$$ Recall that,  $$H^*(B_{\mathbb{Z}_2 \:\oplus\: \mathbb{Z}_2}; \mathbb{Z}_2)= \mathbb{Z}_2[x,y],$$ where both $x$ and $y$ have degree one. The map between the classifying spaces induces a map $$H^*(B_{\mathbb{Z}_2 \:\oplus\: \mathbb{Z}_2}; \mathbb{Z}_2) \rightarrow H^*(B_{\mathbb{Z}_2}; \mathbb{Z}_2)$$ given by $x \rightarrow t, y\rightarrow 0$. In particular, we get a map of fibrations

\begin{equation}\label{diagram}
\begin{tikzcd}
S^m\times\mathbb{C}P^n \arrow{r}{id} \arrow[swap]{d}{} & S^m\times\mathbb{C}P^n \arrow{d}{} \\
P(m,n) \arrow{r}{} \arrow[swap]{d}{}&  S^m \times \mathbb{C}P^n/{\mathbb{Z}_2\oplus \mathbb{Z}_2} \arrow{d}{}\\
B_{\mathbb{Z}_2} \arrow {r}{i} & B_{\mathbb{Z}_2\:\oplus\: \mathbb{Z}_2}.
\end{tikzcd}
\end{equation}

Consequently, we obtain a map between the Leray-Serre spectral sequences associated to the Borel fibrations $$i^*: E_2^{p,q}(\mathbb{Z}_2 \oplus \mathbb{Z}_2)\rightarrow E_2^{p,q}(\mathbb{Z}_2),$$ where by $E_2^{p,q}(G)$ we denote $$E_2^{p,q}(G)\cong H^p\big(B_G; H^q(S^m\times \mathbb{C}P^n;\mathbb{Z}_2)\big).$$
By Lemma \ref{lemma}, the action of $ \pi_1(B_G)$ on $ H^*(\scp;\z2) $ is trivial for $G=\z2 $ or $\zz$. This implies the system of local coefficients is simple. Hence, we get
$$E_2^{p,q}(G)\cong H^p(B_G; \mathbb{Z}_2) \otimes H^q(S^m\times \mathbb{C}P^n;\mathbb{Z}_2).$$
We know that,
$$H^*(S^m\times \mathbb{C}P^n; \mathbb{Z}_2) \cong \mathbb{Z}_2[c,d]/\big\langle c^{2}, d^{\hspace{.3mm}n+1} \big\rangle,$$ where $\deg (c)=m$ and $\deg (d) =2$. Let $d'$ be the differential corresponding to $E_2^{p,q}(\mathbb{Z}_2 )$ and $d''$ be the differential associated to $E_2^{p,q}(\mathbb{Z}_2 \oplus \mathbb{ Z}_2)$.\par
We first  determine all possible non-zero differentials corresponding to the first fibration of diagram (\ref{diagram}). This is done by exploiting the cohomology algebra of $ P(m, n) $. Then we determine all possible non-zero differentials corresponding to the second fibration by comparing  maps between the spectral sequences. In particular, this helps to show that the generator $c \in H^*(S^m, \mathbb{Z}_2)$  is transgressive in  $E^{*,*}(\mathbb{Z}_2 \oplus \mathbb{Z}_2)$.\par
 
Consider the Leray-Serre spectral sequence associated to $$S^m \times \mathbb{C}P^n \stackrel{}{\hookrightarrow} P(m,n) \stackrel{}{\longrightarrow} B_{\mathbb{Z}_2}.$$ 
We have $H^*(P(m,n); \mathbb{Z}_2) \cong \mathbb{Z}_2[a,b]/\big\langle a^{m+1}, b^{n+1} \big\rangle$, where $\deg (a)=1$ and $\deg (b) =2$. From this we observe that the 
differentials in  $E^{*,*}(\mathbb{Z}_2)$ are as follows: $$d'_{m+1}(1\otimes c)= t^{m+1} \otimes 1 \hspace{2mm} \textrm{and }\hspace{2mm} d'_{r}(1\otimes d)=0 \textrm{ for all $r\geq 2$}.$$  Then by the naturality of the Leray-Serre spectral sequence we get   
$$ d''_{m+1}(1\otimes c) = i^* d''_{m+1}(1\otimes c)=d'(1\otimes c) = t^{m+1} \otimes 1.$$
From this it follows that  the generator $c\in H^*(S^m; \mathbb{Z}_2)$  is transgressive in  $E^{*,*}(\mathbb{Z}_2 \oplus \mathbb{Z}_2)$. This gives $$d''_{m+1}(1\otimes c) = f,\hspace{2mm} \textrm{where } f \in H^{m+1}(B_{\z2 \oplus \z2}; \z2)$$ is a non-zero homogeneous polynomial containing the term $ x^{m+1} $. Therefore, $ f $ must be of  the following form:
 $$f= \big(x^{m+1}+ \gamma_1 x^my+\gamma_2x^{m-1}y^2+\dots+\gamma_ny^{m+1}\big) \otimes 1,$$ where $\gamma_i \in \z2$. 
Next we determine the possibilities for $d''_3(1 \otimes d)$. If $d''_3(1 \otimes d)=0$, then $ E^{i,0}_\infty(\mathbb{Z}_2 + \mathbb{Z}_2)$  is non-zero for infinitely many values of $ i $. This contradicts the fact that $S^m \times \mathbb{C}P^n/{\mathbb{Z}_2\oplus \mathbb{Z}_2}$ is a finite-dimensional CW-complex. Thus, \hspace{1mm}$d''_3(1 \otimes d)\neq0$ and this further imposes the condition that $n$ must be odd. 
 Hence, $d''_3(1 \otimes d)=g$, where $g\in H^3(B_{\z2 \oplus \z2}; \z2)$ is a non-zero homogeneous polynomial which does not contain the term $ x^3 $. Therefore, it is of the following form:
  $$g= \big(y^3+ \beta_1 xy^2+\beta_2x^2y\big) \otimes1. $$
Thus, we get that $ 1 \otimes d $ transgresses to $ g $. Recall that, transgression operator commutes with the Steenrod operations, which implies 
$$ d_5''(1\otimes d^2)=0. $$
Using $ g $ we can reduce $ f $ to the following simpler form:
$$f= \big(x^{m+1}+ \alpha_1 x^my+\alpha_2x^{m-1}y^2\big) \otimes1, $$
  where $ \alpha_i \in \z2 $.
Following all the preceding  arguments we get that the spectral sequence $E^{*,*}(\mathbb{Z}_2 \oplus \mathbb{Z}_2)$ degenerates at $E_{m+2}$-page. Hence, $$ E_\infty^{*,*} \cong E_{m+2}^{*,*}.$$\par
We see that $x \otimes 1 \in E_2^{1,0}$ and $y \otimes 1 \in E_2^{1,0}$ are permanent cocycles. Let $i^*$ and $\pi^*$ be the edge maps of the spectral sequence. Let $\tilde{x}=\pi^*(x) \in E_\infty^{1,0} \subseteq H^1\big(S^m \times \mathbb{C}P^n/{\mathbb{Z}_2\oplus \mathbb{Z}_2}; \mathbb{Z}_2\big) $ be determined by $x \otimes 1 \in E_2^{1,0}$. Similarly, let $\tilde{y}=\pi^*(y) \in E_\infty^{1,0}\subseteq H^1(\scp/\zz;\z2)$ be determined by $y \otimes 1 \in E_2^{1,0}$.
Then we get $ \pi^*(f)=\tilde{f}=0$ and $ \pi^*(g) =\tilde{g}=0$.
Note that, $1 \otimes d^2 \in E_2^{0,4}$ is a permanent cocycle and it determines an element in $ E_{\infty}^{0,4}$. Choose $\tilde{z} \in H^4\big(S^m \times \mathbb{C}P^n/{\mathbb{Z}_2\oplus \mathbb{Z}_2};\z2\big)$ such that $i^*(\tilde{z})=d^2$. Since $i^*(\tilde{z } ^{\frac{n+1}{2}})=0$, we get $\tilde{z } ^{\frac{n+1}{2}}+h=0  $, where $  h \in \z2[\tilde{x}, \tilde{y}, \tilde{z }]$ is either a zero polynomial or it is a homogeneous polynomial of degree $ 2n+2 $ with the highest power of $ \tilde{z} $ is less than or equal to $ \frac{n-1}{2} $. So, in this case, we obtain
\begin{align}\label{z2+z2-coho}
H^*(M;\mathbb{Z}_2) \cong \mathbb{Z}_2[\tilde{x},\tilde{y},\tilde{z}]/\big\langle \tilde{f}, \tilde{g}, \tilde{z}^{\frac{n+1}{2}}+h\big\rangle,
\end{align}
where $\deg(\tilde{x})=1$, $\deg(\tilde{y})=1$ and $\deg(\tilde{z})=4$.
Equations (\ref{z4-coho}) and (\ref{z2+z2-coho}) give all possible mod 2 cohomology algebra of orbit spaces of arbitrary free involutions on Dold manifolds.
\end{proof}
\medskip

\begin{remark}
	Recall that, $ P(m,0)= \Bbb RP^m $ and $ P(0,n)= \Bbb CP^n $. The orbit space of a free involution on $\Bbb RP^m  $ is the lens space as any arbitrary free involution on $\Bbb RP^m  $ lifts to a free $ \Bbb Z_4 $-action on $ S^m $ and its cohomology algebra is consistent with Theorem \ref{coho-dold} (1). 
 If $ n $ is odd, then the cohomology algebra of the orbit space of a free involution on $\Bbb CP^n  $  is 
\begin{equation}\label{equ3}
H^*( \mathbb{C}P^n/ \mathbb{Z}_2; \z2) \cong \mathbb{Z}_2[x,y]/\big\langle x^{3}, y^{\hspace{.3mm}\frac{n+1}{2}} \big\rangle,
\end{equation}
 where $\deg (x)=1$ and $\deg (y) =4$ \cite[Corollary 4.7]{msinghprojective}. This is consistent with Theorem \ref{coho-dold} (2).
\end{remark}
\begin{remark}
Conclusions of Theorem \ref{coho-dold} for $ m=1 $ are consistent with \cite[Theorem 1]{morita}.

\end{remark}

\begin{remark}
	Theorem \ref{coho-dold} implies that there are at least two conjugacy classes of free involutions on Dold manifolds. We note that, for an arbitrary free involution on a Dold manifold, the induced action on mod 2 cohomology is always trivial in the following cases:
	
	\begin{enumerate}
		\item[](i) $n$ is of the form $4k+1$ and $ m\geq4 $,
		\item[](ii) $n$ is even,
		\item[] (iii) $m\geq2n+2$.
	\end{enumerate}
	
\end{remark}

\begin{example}
Suppose both $ m$ and   $n $ are odd. Let $ g $ be a generator of $ \Bbb Z_4 $. Define the action of $ \mathbb{ Z}_4 $ on $ \scp $  by
$$g\Big((x_0, x_1, \dots, x_{m-1}, x_{m}), \,[z_0, z_1, \dots, z_{n-1}, z_{n}]\Big) =\Big((-x_1,x_0,\dots,-x_{m},x_{m-1}),\,[\overline{z}_1,\overline{z}_0,\dots,\overline{z}_n,\overline{z}_{n-1}]\Big).$$ Let $ \mathcal{Z} $ denote the orbit space. Further, the preceding action induces a free involution on the Dold manifold $ \p $ with the same orbit space $ \mathcal Z $. The projection $\scp \rightarrow S^m$ induces the map $p:\mathcal{Z}\rightarrow L^m(4,1)$, which is a locally trivial fiber bundle with fiber $\cpn$. Notice that, the system of local coefficients is trivial. Since the projection $ \pi $ has a section, the Leray-Serre cohomology  spectral sequence associated to the fiber bundle $$\Bbb C P^n\stackrel{i}{\hookrightarrow} \mathcal{Z} \stackrel{\pi}{\longrightarrow}L^m(4,1) $$ degenerates at the $ E_2 $-page. Consequently, the map $ i^*:H^*(\mathcal{Z};\z2)\to H^*(\cpn;\z2) $ is onto. Thus, the fiber is totally non-homologous to zero in $ \mathcal{Z} $ with respect to $ \z2 $ (see \cite[Chap. III, 7]{serre}).  Then by \cite [Chap. III, Prop. 9]{serre}, we get $$ H^*(\mathcal{Z};\z2)\cong \mathbb{Z}_2[x, y, z]/\big\langle x^{2}, y^{\hspace{.3mm}\frac{m+1}{2}}, z^{n+1}\big\rangle,$$ where $\deg (x)=1$, $\deg (y) =2$ and $\deg (z)=2$. This realizes Theorem \ref{coho-dold} (1). 
\end{example}
\medskip
\begin{example}
Suppose $n $ is odd. Let $ a $ and $ b $ be  generators of $ \zz $. Define a free  $ \z2 \oplus \z2 $-action on $ \scp $ as follows  
$$a\Big((x_0, x_1, \dots, x_{m-1}, x_{m}),
 \,[z_0, z_1, \dots, z_{n}]\Big) =\Big((-x_0, -x_1, \dots, -x_{m-1}, -x_{m}),\,[z_0, z_1, \dots, z_{n}]\Big)$$
 and
 $$b\Big((x_0, x_1, \dots, x_{m}),
 \,[z_0, z_1, \dots, z_{n-1}, z_{n}]\Big) =\Big((x_0, x_1, \dots, x_{m}),\,[-\overline{z}_1, \overline{z}_0, \dots, -\overline{z}_n, \overline{z}_{n-1}]\Big).$$ Let $\mathcal T$  be the quotient of the above action. Note that, the preceding action induces a free involution on the Dold manifold with the same quotient space $ \mathcal{T} $. The projection $\scp \rightarrow S^m$ induces the map $p:\mathcal T\rightarrow \Bbb R P^m$, which is a locally trivial fiber bundle with fiber the orbit space of $ \cpn $ by the $ \z2 $-action given by $$\big\{[z_0,z_1,\dots,z_n]\mapsto [-\overline{z}_1,\overline{z}_0,\dots,-\overline{z}_n,\overline{z}_{n-1}]\big\}.$$ 
 By equation (\ref{equ3}), we have
 $$H^*( \mathbb{C}P^n/ \mathbb{Z}_2; \z2) \cong \mathbb{Z}_2[x,y]/\big\langle x^{3}, y^{\hspace{.3mm}\frac{n+1}{2}} \big\rangle,$$
where $\deg (x)=1$ and $\deg (y) =4$. It can be seen that the system of local coefficients is trivial. Since the projection $ \pi $ has a section, the Leray-Serre cohomology  spectral sequence associated to the fiber bundle  $${\Bbb C P^n/ \z2}\stackrel{i}{\hookrightarrow} \mathcal T \stackrel{\pi}{\longrightarrow} \Bbb R P^m ,$$ degenerates at the $ E_2 $-page. Consequently, the map $ i^*: H^*(\mathcal T;\z2)\to H^*({\cpn/\z2};\z2) $ is onto. Thus, the fiber is totally non-homologous to zero in $ \mathcal T $  with respect to $ \z2 $ (see \cite[Chap. III, 7]{serre}). Then by \cite [Chap. III, Prop. 9]{serre}, we get $$ H^*(\mathcal T;\z2)\cong \mathbb{Z}_2[x, y, z]/\big\langle x^{3}, y^{\hspace{.3mm}\frac{n+1}{2}}, z^{m+1}\big\rangle,$$ where $\deg (x)=1$, $\deg (y) =4$ and $\deg (z)=1$. This realizes Theorem \ref{coho-dold} (2).
\end{example}
\bigskip
Since the mod 2 cohomology algebras of $\mathbb{R}P^m \times \mathbb{C}P^n$ and  $P(m,n)$ are isomorphic, and both have the same universal cover $S^m \times \mathbb{C}P^n$, the above proof also works for $\mathbb{R}P^m \times \mathbb{C}P^n$. Hence, we get the following corollary.

\begin{corollary}
Let $G=\mathbb{Z}_2$ act freely on $\mathbb{R}P^m \times \mathbb{C}P^n$ 
with a trivial action on mod 2 cohomology.
Then $H^*\big((\mathbb{R}P^m \times \mathbb{C}P^n)/G; \mathbb{Z}_2\big)$ is isomorphic to one of the following graded algebras:
\begin{enumerate}
\item $\mathbb{Z}_2[x,y,z]/ \big\langle x^2, y^{\frac{m+1}{2}}, z^{n+1} \big\rangle,$\\
where $\deg(x)=1$, $\deg(y)=2$, $\deg(z)=2$ and $m$ is odd;
\item$\mathbb{Z}_2[{x},{y},{z}]/\big\langle {f}, {g}, {z}^{\frac{n+1}{2}}+h\big\rangle,$\\
where $\deg({x})=1$, $\deg({y})=1$, $\deg({z})=4$, $n$ is odd,
$$f= \big(x^{m+1}+ \alpha_1 x^my+\alpha_2x^{m-1}y^2\big) \hspace{2mm} \textrm{and}\hspace{2mm} g= \big(y^3+ \beta_1 xy^2+\beta_2x^2y\big) ,$$       
where $\alpha{_i}$, $\beta_j\in \z2$, and  $  h \in \z2[{x}, {y}, {z }]$ is either a zero polynomial or it is a homogeneous polynomial of degree $ 2n+2 $ with the highest power of $ {z} $ is less than or equal to $ \frac{n-1}{2} $. 
\end{enumerate}
\end{corollary} 
\medskip
\section{Application to equivariant maps}\label{sec6}
In this final section, by using the cohomology structure of $ H^*(\p/\z2;\z2) $, we  determine when  there exists a $ \z2 $-equivariant map from $ S^n $ to $ \p $.   
Let $X$ be a compact Hausdorff space with a free involution and consider the antipodal involution on  ${S}^n$. Conner and Floyd \cite{conner-floyd} defined the index of the involution on $ X $ as
$$\textrm{ind}(X) = \max ~ \{~ n ~|~ \textrm{there exists a}~ \mathbb{Z}_2 \textrm{-equivariant map}~ {S}^n \to X \}.$$ Using Borsuk-Ulam theorem one can see that ind($S^n)=n$. Using Characteristic classes we can derive some upper bound for the ind$ (X) $.  Let $w \in H^1(X/G; \mathbb{Z}_2 )$ be the Stiefel-Whitney class of the principal $G$-bundle $X \to X/G$. Conner and Floyd also defined
$$\textrm{{co-ind}}_{\mathbb{Z}_2}(X) = \max ~\{~n~|~ w^n \neq 0 \},$$
and showed that \cite[(4.5)]{conner-floyd} 
$$\textrm{ind}(X)\leq \textrm{{co-ind}}_{\mathbb{Z}_2}(X).$$\par
 Our final result is the following:
\begin{proposition}
If $ n $ is even, then there is no $\mathbb{Z}_2$-equivariant map ${S}^k \to \p$ for $k \geq 2$.
\end{proposition}

\begin{proof}
Let $ X=\p $. Take $ f $ a classifying map $$f : X/\z2 \to B_{\z2}$$ for the principal $\z2$-bundle $X \to X/\z2$. Consider the Borel fibration $X\stackrel{}{\hookrightarrow} X_{\z2} \stackrel{\pi}{\longrightarrow} B_{\z2}$. We know for free actions $ X_{\z2} \simeq X/\z2 $. Let $\eta: X/\z2 \to X_{\z2}$ be the  homotopy equivalence. Then $\pi \circ \eta : X/\z2 \to B_{\z2}$ also classifies the principal $\z2$-bundle $X \to X/\z2$, and hence it is homotopic to $f$. Therefore, it is sufficient to consider the map $$\pi^*: H^1(B_{\z2}) \to H^1(X_{\z2}).$$ The image of the Stiefel-Whitney class of the universal principal $\z2$-bundle $\z2 \hookrightarrow E_{\z2} \longrightarrow B_{\z2}$ is the Stiefel-Whitney class of $X \to X/\z2$. Since we have assumed that $ n  $ is even, only case (1) of Theorem \ref{coho-dold} is applicable.
It follows from the proof of Theorem \ref{coho-dold} (1) that $x \in H^1(X/\z2)$ is the Stiefel-Whitney class with $x \neq 0$ and $x^2=0$. This gives $\textrm{{co-ind}}_{\mathbb{Z}_2}(X) = 1$ and $\textrm{ind}(X)\leq 1$. Hence, there is no $\mathbb{Z}_2$-equivariant map ${S}^k \to \p$ for $k \geq 2$.
\end{proof}

\ack{I thank my Ph.D. advisor Dr. Mahender Singh for suggesting the problem and for many valuable discussions and corrections concerning this paper. I also thank UGC-CSIR for the financial support towards this work.}

\bibliographystyle{plain}

\end{document}